\documentclass[a4paper,12pt]{article}
\usepackage{latexsym,amssymb,amsfonts,amsmath,amsthm}
\usepackage[dvips]{graphicx}
\newtheorem{lemma}{Lemma}
\newtheorem{corollary}{Corollary}
\newtheorem{proposition}{Proposition}

\newtheorem{definition}{Definition}

\newcommand{\bs}[1]{\boldsymbol{#1}}

\title{{\Large {\bf Spectral properties of weighted line digraphs   
}
}}
\author{ 
{\small 
Etsuo Segawa,$^{1}$ 
\footnote{email address e-segawa@m.tohoku.ac.jp
}\quad 
}\\ 
{\scriptsize $^{1}$ 
Graduate School of Information Sciences, Tohoku University
}\\
{\scriptsize 
Aoba, Sendai 980-8579, Japan
} \\
} 
\date{\empty }
\begin{document}
\maketitle

\par\noindent
\begin{small}
\par\noindent
{\bf Abstract}. 
In this paper, we treat some weighted line digraphs which are induced by a connected and undirected graph. 
For a given graph $G$, the adjacency matrix of the weighted line digraph $W$ is determined by a boundary operator 
from an arc-based space to a vertex-based space. 
We see that depending on the boundary operator and the Hilbert spaces, 
$W$ has different kind of an underlying stochastic transition operator. 
As an application, we obtain the spectrum of the positive support of cube of the Grover matrix in a large girth of the graph.  
\footnote[0]{
{\it Key words and phrases.} 
Spectral mapping theorem, Szegedy walk, Quantum graph, Positive support of the cube of the Grover walk
}

\end{small}

\setcounter{equation}{0}

\section{Introduction}
Quantum walk was proposed as iterations of a unitary operator on some discrete space~\cite{Gudder, Severini}. 
Quantum walks (QWs) have been intensively studied since its efficiency of spatial quantum search were shown i.g., \cite{Am} and see its reference for more detail. 
Now QWs are investigated from various fields, for example, implementation in experiments, condensed matter physics, probability theories, 
and graph theories and so on. 

For a given graph $G$, the spectral mapping theorem of quantum walks have been founded by Szegedy~\cite{Sze}. 
In this paper, we consider this spectral map to three kinds of quantum walks. 
First we consider a special class of quantum walks, which can be constructed on any (undirected) graphs and its spectrum has a
relation to a random walk. This class of QW is called the Szegedy walk including the Grover walk. 
We refine the original one to see a relation to a graph laplacian and the Grover walks which work well in the spacial quantum search for hypercube, 
Johnson graph and 
finite $d$-dimensional lattice and so on, see \cite{Am} and its reference. 
The original Szegedy walk introduced by \cite{Sze} is equivalent to square of the Szegedy walk treated here. 

Secondly we treat a quantum graph walk introduced by \cite{HKSS:YMJ}. 
Recently, a relationship between this quantum walk and the quantum graph~\cite{Exner}, 
which is a system of a linear Schr{\"o}dinger equations with boundary conditions. 
has been found~\cite{HKSS:YMJ,Tan}. 
In this paper, we develop this study by connecting the spectrum of the original quantum graph and this quantum walk. 
To this end, we apply the spectral map. 

Finally we also consider the cube of positive support of the Grover walk. 
There are some trials to apply QW to graph isomorphism problems. 
For a matrix $M$, the positive support of $M$, $M^+$, is denoted by 
	\[ (M^+)_{i,j}= \begin{cases} 1 & \text{: $(M)_{i,j}\neq 0$} \\ 0 & \text{: otherwise} \end{cases} \]
It is suggested that the spectra of $(U^3)^+$, which is the positive support of the cube of the Grover matrix $U^3$, 
outperforms distinguishing strongly regular graphs in \cite{Emms}. 

This paper is organized as follows. 
In Section 2, we construct quantum walks on $\ell^2(m_A,A)$, where $A$ is the set of the symmetric arcs and $m_A$ is a positive real valued measure. 
In Section 3, first, we present the spectral mapping theorem from $C^0(G)$ to $C^1(G)$, which is the key in our paper. 
secondly, we take three examples Szegedy walk~\cite{HKSS:JFA,Sze}, quantum graph walk~\cite{Exner,HKSS:YMJ} and 
the positive supports of the Grover walks~\cite{Emms,HKSS:JMI,HKSS:PJMI} as applications.  
\section{Szegedy walk: reconsideration}
Let $G=(V(G),E(G))$ be a connected and locally finite graph, where $V(G)$ is the set of vertices and $E(G)$ is the set of all edges. 
We set $A(G)=\{(u,v): \{u,v\}\in E(G)\}$ as the set of all arcs of $G$, and the origin vertex and terminal one of $e=(u,v)\in A(G)$ 
are denoted by $o(e)=u$ and $t(e)=v$. 
The inverse arc $e=(u,v)$ are denoted by $\bar{e}=(v,u)$. 
We take 
	\[C^0(G)=\{f:V\to \mathbb{C}\},\;\; C^1(G)=\{\psi:A\to \mathbb{C}\} \] 
as the space of all functions on $V(G)$ and $A(G)$, respectively. 
We prepare positive real valued functions $m_A: C^1(G) \to C^1(G)$ and $m_V: C^0(G)\to C^0(G)$. 
For given such $m_V$ and $m_A$, we also set $\ell^2(m_V;V)$ and  $\ell^2(m_A;A)$ as the square-summable functions with respect to the inner product 
	\begin{align} 
        \langle f_1,f_2\rangle_{m_V,V} &= \sum_{u\in V(G)} \overline{f_1(u)}f_2(u)m_V(u),  \\
	\langle \psi_1,\psi_2\rangle_{m_A,A} &= \sum_{e\in A(G)} \overline{\psi_1(u)}\psi_2(e)m_A(e),   
        \end{align}
respectively. 
From now on, we construct the time evolution of a quantum walk on $\ell^2(m_A;A)$. 
We set a complex valued function $w: C^1(G) \to C^1(G)$ to define boundary operator $d: C^1(G)\to C^0(G)$: 
	\begin{equation}\label{Eq:d} (d\psi)(u)=\sum_{e:t(e)=u}\overline{w(\bar{e})}\psi(e). \end{equation}
This operator takes just terminus one to get a unitary operator on $\ell^2(m_A,A)$, which is different from a ``usual" boundary operator. 
We introduce coboundary operator $d^*: C^0(G) \to C^1(G)$ satisfying $\langle d^*f, \psi \rangle_{m_A,A}=\langle f,d\psi \rangle_{m_V,V}$:  
	\begin{align*}
        \langle f, d\psi \rangle_{m_V,V} &= \sum_{u\in V} \overline{f(u)}\sum_{e:t(e)=u}\psi(e)\overline{w(\bar{e})} m_V(u) \\
        	&= \sum_{e\in A} \psi(e) \overline{f(t(e))} \frac{m_V(t(e))\overline{w(\bar{e})}}{m_A(e)} m_A(e) \\
                &= \langle d^*f, \psi \rangle_{m_A,A}.
	\end{align*}
So we have \begin{equation}\label{Eq:d*} (d^*f)(e)=f(t(e))\frac{m_V(t(e))w(\bar{e})}{m_A(e)}. \end{equation} 
Put $S: C^1(G)\to C^1(G)$ by $(S\psi)(e)=\psi(\bar{e})$. 
We define $U: C^1(G)\to C^1(G)$ as follows. 
	\begin{equation}\label{Eq:U} U=S(2d^*d-1_A). \end{equation}
	\begin{definition} (Szegedy walk on $G(V,E)$) \\
        When $U$ defined in Eq.~(\ref{Eq:U}) is norm conservative in $\ell^2(A;m_A)$ i.e., $||U\psi||_{m_A,A}=||\psi||_{m_A,A}$ for any $\psi\in \ell^2(m_A;A)$, we define the quantum walk as follows: 
        The quantum walk with a unit initial state $\psi_0\in \ell^2(m_A;A)$ is defined by the following sequence of probability distributions 
        $\{\mu_n^{(\psi_0)}\}_{n\in \mathbb{N}}$, 
        where $\mu_n^{(\psi_0)}:A\to [0,1]$ is determined by 
        \begin{equation}\label{Def:dist}
        	\mu_n^{(\psi_0)}(e)=|\psi_n(e)|^2 m_A(e).
        \end{equation}
        Here $\psi_n$ is the $n$-th iteration of $U$ i.e., $\psi_n=U\psi_{n-1}$ $(n\geq 1)$. 
        In particular, we define the finding probability of a quantum walker at time $n$ in location of $u\in V$, $\nu_n: V\to [0,1]$, by 
        	\[ \nu_n^{(\psi_0)}(u)=\sum_{e:t(e)=u}\mu_n^{(\psi_0)}(e).  \]
	\end{definition} 
We provide two ways which give the operator $U$ defined in Eq.~(\ref{Eq:U}) preserving the norm with respect to $\langle \cdot,\cdot\rangle_{m_A,A}$. 
Both of them have some underlying dynamics on $C^0(G)$; the first one is a random walk while the second one is a laplacian. 
\begin{enumerate}
\item Setting 1
\begin{enumerate}
\item $\sum_{e:o(e)=u}|w(e)|^2=1$ for all $u\in V$;
\item $m_V(u)=1,\;m_A(e)=1$ for all $u\in V$ and $e\in A$ . 
\end{enumerate}
\item Setting 2
\begin{enumerate}
\item $\sum_{e:o(e)=u}w(e)=1$ for all $u\in V$; 
\item $m_A(e)=m_V(o(e))w(e)=m_V(t(e))w(\bar{e})$ for all $e\in A$. 
\end{enumerate}
\end{enumerate}
We call the condition (2) in Setting 2 ``extended detailed balanced condition". 
We can easily check that both of the coboundary operators $d^*$ in the setting of (1) and (2) are isometric. 
	\begin{proposition}
	Both settings (1) and (2) imply that $U$ preserves the norm. 
	\end{proposition}
\begin{proof}
It is easy to check that both settings (1) and (2) provide 
        \begin{align} 
        	\langle S\psi_1,S\psi_2 \rangle_{m_A,A} &= \langle \psi_1,\psi_2 \rangle_{m_A,A}. \label{Eq:SS}\\
                dd^* &= 1_V \label{Eq:dd}
        \end{align} 
By Eq.~(\ref{Eq:SS}), it holds that 
	\begin{align}
        \langle U\psi,U\psi \rangle_{m_A,A} &= \langle S(2d^*d-1_A)\psi,S(2d^*d-1_A)\psi \rangle_{m_A,A}=\langle (2d^*d-1_A)\psi,(2d^*d-1_A)\psi \rangle_{m_A,A} \notag \\
        	&= 4 \langle d^*d\psi,d^*d\psi \rangle_{m_A,A}-2\langle \psi,d^*d\psi \rangle_{m_A,A}-2 \langle d^*d\psi,\psi \rangle_{m_A,A}+\langle \psi,\psi \rangle_{m_V,V} \label{Eq:pro1}
        \end{align}
From Eq.~(\ref{Eq:dd}), the inner product of the first term of RHS is rewritten by
 	\begin{align}
         \langle d^*d\psi,d^*d\psi \rangle_{m_A,A} &= \langle d\psi,d\psi \rangle_{m_V,V} \notag \\
         	&= \langle \psi,d^*d\psi \rangle_{m_V,V}=\langle d^*d\psi,\psi \rangle_{m_V,V}. \label{Eq:pro2}
	\end{align}
Inserting Eqs.~(\ref{Eq:pro2}) into (\ref{Eq:pro1}), we have $||U\psi||^2_{m_A,A}=||\psi||^2_{m_A,A}$. 
\end{proof}
We take $U_1$ and $U_2$ as the time evolution of quantum walks with respect to the inner products under the settings of 1 and 2, respectively. 
Equivalently, $U_1$ and $U_2$ are expressed by 
	\begin{align}
        (U_1\psi)(e)&= \sum_{f:o(e)=t(f)} \left(2w_1(e)\overline{w_1(\bar{f})}-\delta_{e,\bar{f}}\right)\psi(f), \label{Eq:U1}\\
        (U_2\psi)(e)&= \sum_{f:o(e)=t(f)} \left(2w_2(\bar{f})-\delta_{e,\bar{f}}\right)\psi(f), \label{Eq:U2}
        \end{align}
where 
\begin{align}
(d_j\psi)(u) &= \sum_{e:t(e)=u} \psi(e)\overline{w_j(\bar{e})} \;\;(j=1,2)\\
(d_j^*f)(e) &= \begin{cases}
		f(t(e))w_1(\bar{e}) & \text{: $j=1$,} \\
                f(t(e)) & \text{: $j=2$.}
        	\end{cases}
\end{align}
Here $w_j: A\to \mathbb{C}$ $(j=1,2)$ satisfy settings 1 (a) and 2, respectively. 
These two quantum walks are essentially the same in the following meaning: 
\begin{proposition}
Assume that $w_1^2(e)=w_2(e)>0$ $(e\in A)$. 
For $\psi\in C^1(G)$, 
let $\psi_n^{(\psi)}=U_1^n\psi$, $\widetilde{\psi}_n^{(\psi)}=U_2^n\psi$ and $\mu_n^{(\psi)}$ and $\widetilde{\mu}_n^{(\psi)}$ be 
the distributions of QWs at time $n$ defined in Eq.~(\ref{Def:dist}) under the settings 1 and 2, respectively. 
Then we have 
\begin{enumerate}
\item $\widetilde{\psi}_n^{(\psi_0)}=\mathfrak{D}^{-1/2}\psi_n^{(\mathfrak{D}^{1/2}\psi_0)}$
\item $\widetilde{\mu}_n^{(\psi_0)}=\mu_n^{(\mathfrak{D}^{1/2}\psi_0)}$
\end{enumerate}
\end{proposition}
\begin{proof}
For the proof of part 1, comparing Eq.~(\ref{Eq:U1}) with Eq.~(\ref{Eq:U2}) immediately provides that 
if $w_1^2(e)=w_2(e)\in \mathbb{R}$, then we have 
	\begin{equation}\label{Eq:U2U1}
        U_2=\mathfrak{D}^{-1/2} U_1 \mathfrak{D}^{1/2}, 
        \end{equation}
where $(\mathfrak{D}\psi)(e)=m_A(e)\psi(e)$. 
Equation (\ref{Eq:U2U1}) implies part 1. 
For the proof of part 2, using Eq.~(\ref{Eq:U2U1}),
	\begin{align*} 
        \widetilde{\mu}_n^{(\psi_0)}(e) &= |\widetilde{\psi}_n(e)|^2m_D(e) 
          =\left| \mathfrak{D}^{1/2} (\mathfrak{D}^{-1/2}U_1^n\mathfrak{D}^{1/2}\psi_0)(e) \right|^2 
          =|(U_1^n\mathfrak{D}^{1/2}\psi_0)(e)|\\
          &= \mu_n^{(\mathfrak{D}^{1/2}\psi_0)}(e)
        \end{align*}
\end{proof}

\section{Spectral map}
\subsection{Spectral map}
For a given boundary operator $d$, let us define an operator on $C^1(G)$ by 
	\begin{equation}
        W=S(c d^*d-1_A),
        \end{equation}
We assume that the following value $c'$ is constant with respect to the vertices of the given graph: 
	\[ c'=\sum_{e:t(e)=u}\frac{m_V(t(e))|w(\bar{e})|^2}{m_A(e)}\;\;\mathrm{for\;any} u\in V. \]
This assumption implies $dd^*=1_V$. We also assume $cc'\neq 1$ for some technical reason. 
Remark that in the present stage, $W$ has no restriction of the norm conservation. 
We provide examples by taking appropriate functions $m_A$, $m_V$, $w$ and parameter $c\in \mathbb{C}$ in the later subsection. 
We set 
	\[ \mathcal{H}_\pm =\{ \psi\in \ell^2(A): \psi(\bar{e})=\pm \psi(e) \;\;(e\in A) \}. \]
The following lemma is the key in this paper, which is an extended version of Proposition~1 in \cite{HKSS:JFA}. 
\begin{lemma}\label{keylemma}
Let $G$ be a finite graph. 
Put $\varphi(x)=(x+(cc'-1)x^{-1})/c$ and $\mathcal{L}=d^*(C^1(G))+Sd^*(C^1(G))\subseteq C^1(G)$. 
Then we have $W(\mathcal{L})=\mathcal{L}$  and 
	\begin{align} 
        \sigma(W|_{\mathcal{L}}) &= 
		\begin{cases}
        	\varphi^{-1}\left(\sigma (d^*Sd)\setminus \{\pm c'\}\right) \cup \{ \pm (cc'-1) \} & \text{: $\pm c'\in \sigma(dSd^*)$}\\
        	\varphi^{-1}\left(\sigma (d^*Sd)\setminus \{c'\}\right) \cup \{ cc'-1 \} & \text{: $c'\in \sigma(dSd^*)$}\\
        	\varphi^{-1}\left(\sigma (d^*Sd)\setminus \{- c'\}\right) \cup \{ -(cc'-1) \} & \text{: $- c'\in \sigma(dSd^*)$}\\
        	\varphi^{-1}\left(\sigma (d^*Sd)\right) & \text{: otherwise}
		\end{cases} \label{specmap} \\
        \sigma(W|_{\mathcal{L}^\bot}) &= 
        	\begin{cases}
                \{1,-1\} & \text{: $\mathcal{H}_-\cap \mathrm{ker}(d)\neq \{\bs{0}\}$, $\mathcal{H}_+\cap \mathrm{ker}(d)\neq \{\bs{0}\}$}\\
                \{1\} & \text{: $\mathcal{H}_-\cap \mathrm{ker}(d)\neq \{\bs{0}\}$, $\mathcal{H}_+\cap \mathrm{ker}(d)= \{\bs{0}\}$}\\
                \{-1\} & \text{: $\mathcal{H}_-\cap \mathrm{ker}(d)= \{\bs{0}\}$, $\mathcal{H}_+\cap \mathrm{ker}(d)\neq \{\bs{0}\}$}\\
                \emptyset & \text{: otherwise.} \label{genesis}
                \end{cases}
        \end{align}
The eigenfunction for $\lambda \in \sigma(W|_\mathcal{L})$, $\psi_\lambda\in C^1(G)$, 
can be expressed by using the eigenfunction of $dSd^*$ for $\varphi(\lambda)\in \sigma(dSd^*)$, $f_{\varphi(\lambda)}\in C^0(G)$, as follows: 
	\begin{equation}\label{eigenfunc} 
        \psi_\lambda
        =\begin{cases}
        (d^{*}-\lambda Sd^{*})f_{\varphi(\lambda)} & \text{: $\lambda\neq \{\pm (cc'-1)\}$} \\
        \pm d^*f_{\varphi(\lambda)} & \text{: $\pm c'\in \sigma(dSd^*)$, $\lambda\in \{\pm (cc'-1)\}$} 
        \end{cases}
        \end{equation}
\end{lemma}
\begin{proof}
From the properties of $d$ and $S$, it holds that for any $n\in \mathbb{N}$, 
	\begin{equation}\label{Eq:close}
	\begin{bmatrix} W d^* & W Sd^* \end{bmatrix}
        	= \begin{bmatrix} d^* & Sd^* \end{bmatrix} \widetilde{T}. 
	\end{equation}
Here we put 
	\[ \widetilde{T}=\begin{bmatrix} 0 & -1_V \\ (cc'-1)1_V & cdSd^* \end{bmatrix}\]
By multiplying arbitrary ${}^T[f_1, f_2]$ to both sides of Eq.~(\ref{Eq:close}), we have $W(\mathcal{L})\subset \mathcal{L}$. 
For any $h\in \mathcal{L}$, there exist $f,g\in C^0(G)$ such that $h=d^*f+Sd^*g$. By Eq.~(\ref{Eq:close}) 
	\[ W(d^* (c(dSd^*)f+g) - Sd^* f)=(cc'-1)h, \]
which implies $W(\mathcal{L})\supset \mathcal{L}$ since $cc'\neq 1$. Thus $W(\mathcal{L})= \mathcal{L}$.  
From now on, we choose a pair of nonzero functions $(f,g)\in C^0(G)\times C^0(G)$ so that 
	\begin{equation}\label{Eq:f1f2} 
        	\widetilde{T} \begin{bmatrix} f \\ g \end{bmatrix}
        	=\lambda \begin{bmatrix} f \\ g \end{bmatrix}. 
        \end{equation}
Eq.~(\ref{Eq:f1f2}) is equivalent to 
	\begin{align}
        g &= -\lambda f \label{Eq:2}\\
        \left(\varphi(\lambda)-dSd^*\right)f & = 0 \label{Eq:1}
        \end{align}
From Eqs.~(\ref{Eq:f1f2}) and (\ref{Eq:2}), we have 
	\begin{equation}\label{Eq:vec}
        W(1_A-\lambda S)d^*f=\lambda (1_A-\lambda S)d^*f
        \end{equation} 
By the way, let $\nu\in \mathbb{R}$ and $f\in C^0$ be an eigenvalue and its eigenfunction $dSd^*$, which is a self adjoint operator. 
Set $\Pi=(1/c')d^*d$. $\Pi$ is a projection onto $d^*(C^0(G))$.  Taking $d^*f=\psi$, we have 
	\[ |\nu|^2 ||f||^2 = ||dSd^*f||^2=||  d\Pi S \Pi \psi ||^2 = c'|| \Pi S \Pi \psi ||^2 \leq c'|| \Pi\psi ||^2 = (c')^2 ||f||^2.  \]
Thus $\nu\leq c'$. The fourth equality holds if and only if $Sd^*f=d^*f$ or $Sd^*f=-d^*f$, 
which implies that 
	\[ \nu=\pm c' \mathrm{ \; if \; and \; only \; if \; } d^*f\in \mathcal{H}_\pm \] 
Now from Eq.~(\ref{Eq:vec}), we consider the following two cases. 
\begin{enumerate}
\item $(1_A-\lambda S)d^*f=0$ case. From the above observation, this situation is equivalent to ``$\lambda=1$ and $c'\in \sigma(T)$" or ``$\lambda=-1$ and $-c'\in \sigma(T)$", and $f\in \mathrm{ker}(dSd^*\pm c')$. 
Remark that $Wd^*f=\pm (cc'-1)d^*f$ and $\varphi^{-1}(\pm c')= \{\pm 1, \pm (cc'-1)\}$. 
So we have if $\pm c'\in \sigma(dSd^*)$, then $\pm (cc'-1)\in \varphi^{-1}(\pm c')\setminus \{\pm 1\}\subset \sigma(W|_\mathcal{L})$. 

\item $(1_A-\lambda S)d^*f\neq 0$ case. 
This situation implies $\lambda\in \sigma(U)$. Equation~(\ref{Eq:1}) implies that 
	\begin{equation}\label{Eq:val} 
        \lambda\in \varphi^{-1}(\sigma (d^*Sd)\setminus \{\pm c'\}). 
        \end{equation}
In this case, the eigenfunction $(1_A-\lambda S)d^*f\notin \mathcal{H}_\pm$ but $(1_A-\lambda S)d^*f \in \mathcal{H}_+ + \mathcal{H}_-$. 
\end{enumerate}
$\mathrm{dim}(\mathcal{H}_+\cap \mathrm{L})=|V|$
Combining cases (1) and (2), we have Eq.~(\ref{specmap}) and Eq.~(\ref{eigenfunc}).

In the following, we show Eq.~(\ref{genesis}). Since $\mathcal{L}^\bot=\mathrm{ker}(d)\cap \mathrm{ker}(dS)$, we have  
	\begin{equation}
        W(\psi+S\psi)=-(\psi+S\psi)\in \mathcal{H}_- \mathrm{\;and\;} W(\psi-S\psi)=\psi+S\psi\in \mathcal{H}_+. 
        \end{equation}
We set $\mathcal{C}_\pm=\mathcal{L}^\bot \cap \mathcal{H}_{\mp}$. 
We have  
	\begin{equation}\label{Eq:gen} W\psi_\pm=\pm \psi_\pm,\;\;\mathrm{for\;any\;} \psi_\pm\in \mathcal{C}_\pm. \end{equation} 
For $\psi\in \mathcal{H}_\mp \cap \mathrm{ker}(d)$, $dS\psi=\pm d\psi=0$ holds which implies $\mathcal{H}_\mp \cap \mathrm{ker}(d)=\mathcal{C}_\pm$. 
It is completed the proof
\end{proof}

\subsection{Example 1: Szegedy walk}
We consider the spectral map of the Szegedy walk treated in Sect.~2. 
The parameter $c$ and functions $m_A$, $m_V$ and $w$ are in the following table: 
\begin{center}
\begin{tabular}{|c||c|c|}\hline
    & $U_1$ & $U_2$ \\ \hline\hline
$c$ & $2$ & $2$ \\ \hline
$w$ & $\sum_{e:o(e)=u}|w_1(e)|^2=1$ & $\sum_{e:o(e)=u}w_2(e)=1$ and $w_2>0$ \\ \hline
$m_A$ and $m_V$ & $m_A(e)=1$, $m_V(u)=1$ & $m_A(e)=m_V(t(e))w(\bar{e})=m_V(o(e))w(e)$ \\ \hline
\end{tabular}
\end{center}
Let a self adjoint operator on $C_0(G)$ be 
	\[ (Tf)(u)=\sum_{e:o(e)=u}w_1(e)\overline{w_1(\bar{e})} f(t(e)) \] 
and a Laplacian operator be 
	\[ (Lf)(u)=\sum_{e:o(e)=u}w_2(e) f(t(e))-f(u). \]
\begin{proposition}
Put $\varphi_1(x)=(x+x^{-1})/2$, $\varphi_2(x)=(x+x^{-1})/2-1$. 
\begin{align}
\sigma(U_1) &= 
	\begin{cases}
        \varphi^{-1}_1(\sigma(T)) \cup \{\pm 1\} & \text{: $|E|-|V|+\bs{1}_{\{ 1\in \sigma(T) \}}>0$, $|E|-|V|+\bs{1}_{\{ -1 \in \sigma(T) \}}>0$ } \\
        \varphi^{-1}_1(\sigma(T)) \cup \{1\} & \text{: $|E|-|V|+\bs{1}_{\{ 1\in \sigma(T) \}}>0$, $|E|-|V|+\bs{1}_{\{ -1 \in \sigma(T) \}}\leq 0$ } \\
        \varphi^{-1}_1(\sigma(T)) \cup \{-1\} & \text{: $|E|-|V|+\bs{1}_{\{ 1\in \sigma(T) \}}\leq 0$, $|E|-|V|+\bs{1}_{\{ -1 \in \sigma(T) \}}>0$ } \\
        \varphi^{-1}_1(\sigma(T))  & \text{: $|E|-|V|+\bs{1}_{\{ 1\in \sigma(T) \}}\leq 0$, $|E|-|V|+\bs{1}_{\{ -1 \in \sigma(T) \}}\leq 0$ } 
	\end{cases} \label{Eq:Spec1} \\
\sigma(U_2) &= 
	\begin{cases}
        \varphi^{-1}_2(\sigma(L))  & \text{: $G$ is a tree} \\
        \varphi^{-1}_2(\sigma(L)) \cup \{1\} & \text{: $G$ has just one cycle and not bipartite} \\
        \varphi^{-1}_2(\sigma(L)) \cup \{1\} \cup \{-1\} & \text{: otherwise} 
        \end{cases}\label{Eq:Spec2}
\end{align}

\end{proposition}
\begin{proof}
By a simple observation, we can see that 
	\begin{align}\label{Eq:RW}
        d_1^*Sd_1 &= T,\;\;d_2^*Sd_2 = L+1_V. 
        \end{align}
We should remark that $L+1_V={}^TP$, where $P$ is a transition matrix of random walk: $(Pf)(u)=\sum_{e:t(e)=u}w_2(e)f(e)$. 
In the setting of $U_2$, since $w_2(\cdot)$ is reversible, then $\sigma({}^TP)=\sigma(P)$ holds, which implies 
$\sigma({}^TP)=\sigma(T)$ when $w_2=w_1^2$. From the reversibility of the stochastic process described by transition matrix $P$, 
it holds that $m_1=1$ and $m_{-1}=\bs{1}_{\{ \mathrm{G\;is\;bipartite}\}}$. 
From now on, we consider the dimensions of $\mathcal{C}_\pm$. 
By Eq.~(\ref{eigenfunc}) and $\varphi^{-1}(\nu)\in \{ \lambda, \bar{\lambda} \}$ with $\nu=\mathrm{Re}(\lambda)$, we observe
\begin{align*}
	\psi_\lambda+\lambda\psi_{\bar{\lambda}} \in
        	\begin{cases}
                \mathcal{H}_-\setminus \{\bs{0}\} & \text{: $\nu \neq  -1$} \\
                \{\bs{0}\} & \text{: $\nu =-1$}
                \end{cases}\\
        \psi_\lambda-\lambda\psi_{\bar{\lambda}}\in  
        	\begin{cases}
                \mathcal{H}_+\setminus \{\bs{0}\} & \text{: $\nu \neq 1$} \\
                \{\bs{0}\} & \text{: $\nu = 1$}
                \end{cases}
\end{align*}
Therefore
	\begin{equation}\label{Eq:dim}
        \mathrm{dim}(\mathcal{L}\cap \mathcal{H}_+)=|V|-m_{-1} \mathrm{\;and\;} \mathrm{dim}(\mathcal{L}\cap \mathcal{H}_-)=|V|-m_{1}.
        \end{equation}
Remarking $\mathrm{dim}(\mathcal{H}_\pm)=|E|$, by Eq.~(\ref{Eq:dim}),
	\begin{align}
        \mathrm{dim}(\mathcal{C}_+) &= \mathrm{dim}(\mathcal{H}_-)-\mathrm{dim}(\mathcal{L}\cap \mathcal{H}_-)=|E|-|V|+m_{1} \label{Eq:dC+}\\
        \mathrm{dim}(\mathcal{C}_-) &= \mathrm{dim}(\mathcal{H}_+)-\mathrm{dim}(\mathcal{L}\cap \mathcal{H}_+)=|E|-|V|+m_{-1} \label{Eq:dC-}
        \end{align}
We should remark that 
	\begin{equation}\label{Eq:tree} |E|-|V| \begin{cases}  =-1 & \text{: $G$ is a tree} \\ = 0 & \text{: $G$ has one cycle} \\ \geq 1  & \text{: $G$ has at least $2$ cycles}\end{cases} \end{equation}
Combining Eqs.~(\ref{Eq:dC+}), (\ref{Eq:dC-}) and Eq.~(\ref{Eq:gen}) with Eq.~(\ref{Eq:tree}) leads to Eq.~(\ref{genesis}). 

Then Lemma~\ref{keylemma} completes the proof. 
\end{proof}

\subsection{Example 2: Quantum graph walk}
Quantum graph is a system of linear Schr{\"o}diner equations on a metric graph with boundary conditions at each vertex. 
In the metric graph, each edge has a euclidean length, and there exists a one-dimensional plain wave on each edge. 
On $e\in A(G)$ whose euclidean length is $L_e=L_{\bar{e}}$, we determine $x\in [0,L_e]$ as the distance from $o(e)$. 
The plain wave on $e\in A(G)$ at $x$ is described by using some complex numbers $a_e,b_e$ as follows. 
	\begin{equation}\label{plain} 
        \phi_e(x)= a_e e^{ikx}+b_e e^{-ikx}, \;\; (x\in[0,L_e]) 
        \end{equation} 
Since the plain wave on each edge is independent of the arc direction, we should impose the following condition to $\phi_e$. 
	\begin{equation} \label{usiromae} 
        \phi_{\bar{e}}(x)=\phi_e(L_e-x)\;\;  e\in A(G)
        \end{equation}
Combining Eqs.~(\ref{plain}) and (\ref{usiromae}), $\phi_e(x)$ is rewritten by 
	\begin{equation}
        \phi_e(x)= a_e e^{ikx}+a_{\bar{e}} e^{ik(L-x)}. 
        \end{equation}
The boundary conditions at each vertex are as follows: 
	\begin{enumerate}
        \item For every $u\in V$ and for all $e_1,\dots,e_\kappa \in A(G)$ with $u=o(e_1)=o(e_2)$, there exists $\theta_u\in \mathbb{C}$ such that 
        \[ \phi_{e_1}(0)=\cdots=\phi_{e_\kappa}(0)=\theta_u. \] 
        \item For every $u\in V$, there exists $\alpha_u\in [0,\infty]$ such that 
        \[ \sum_{f:o(f)=u} \phi'_f(0)=\alpha_u \theta_u. \]
        \end{enumerate}
\begin{definition}
        Let $\psi_n\in C^1(G)$ $(n\in\mathbb{N})$ be 
        determined by the iterations of $\widetilde{U}$ i.e., $\psi_n=\widetilde{U}\psi_{n-1}$, where
	\[ (\widetilde{U}\psi)(e)= e^{ikL_e}\sum_{f:t(f)=o(e)} (\tau_{o(e)}-\delta_{f,\bar{e}}) \psi(f).  \]
	Here 
	\[ \tau_u=\frac{2}{\mathrm{deg}(u)+i\alpha_u/k}. \]
	The quantum graph walk is sequence of distributions $\{\mu_n\}_{n\in \mathbb{N}}$ obtained by 
        \[ \mu_n(e)=|\psi_n(e)|^2. \]
\end{definition}
Note that $\widetilde{U}$ is a unitary operator i.e., $\widetilde{U}\widetilde{U}^*=\widetilde{U}^*\widetilde{U}=1_A$. 
Set $\psi\in C^1(G)$ with $\psi(e)=a_e$. If $(a_e)_{e\in A}$ satisfies the boundary conditions (1) and (2) with $\psi\neq 0$, then we say that the quantum graph has non-trivial solution. 
\begin{proposition}\cite{HKSS:YMJ}
	The quantum graph walk has a non-trivial solution if and only if $1\in \sigma(\widetilde{U})$. 
\end{proposition}
From now on we restrict ourselves to consider the following simple case to make $\widetilde{U}$ be able to apply the spectral map of Lemma~\ref{keylemma}. 
\begin{lemma}
        If $L_e=L$ for any $e\in A$, $\mathrm{deg}(u)=\kappa$ and $\alpha_u=\alpha$ for any $u\in V$, then we have 
        \begin{equation}
        \widetilde{U}=e^{ikL}S(cd^*d-1_A), 
        \end{equation}
        where \[ c=\frac{2\kappa}{\kappa+iq(k)},\;w(e)=\frac{1}{\sqrt{\kappa}},\;m_V(u)=1,\;m_A(e)=1.  \]
        Here we put $q(k)=\alpha/k$. 
\end{lemma}
Remark that in this setting the boundary operator $d$ is given by 
	\[ (d\psi)(u)=\frac{1}{\sqrt{\kappa}}\sum_{e:t(e)=u}\psi(e). \]
We can easily check that $dd^*=1_V$ which implies $c'=1$. 
\begin{proposition}Let $P$ be a stochastic transition operator of the simple random walk on $G(V,E)$. 
Then the spectrum of the quantum graph walk under the setting of $L_e=L$ for any $e\in A$, $\mathrm{deg}(u)=\kappa$ and $\alpha_u=\alpha$ for any $u\in V$, we have 
	\begin{multline}
        e^{-ikL}\sigma(\widetilde{U})=e^{i\mathrm{arg}(\tau)} \varphi_1^{-1}\left( \frac{\kappa}{\sqrt{\kappa^2+q^2(k)}}\sigma(P) \right) \\
        \cup 
        \begin{cases}
        \{\emptyset\} & \text{: ``$G$ is a tree" or ``$q(k)\neq 0$ and $G$ has at most one cycle".}\\ 
        \{1\} & \text{: ``$q(k)=0$ and $G$ has just one cycle whose length is odd".}\\
        \{\pm 1\} & \text{: otherwise.}
        \end{cases}
        \end{multline}
\end{proposition}
\begin{proof}
Eq.~(\ref{Eq:1}) implies 
	\begin{equation}\label{Eq:rew}
        \mathrm{det}(\lambda^2-cJ\lambda+c-1)=0, 
        \end{equation}
where $J=dSd^*$ is expressed by 
	\[ (Jf)(u)=\sum_{e:t(e)=u}\frac{1}{\sqrt{\mathrm{deg}(u)\mathrm{deg}(o(e))}}. \]
Since $P$ is reversible, $\sigma(P)=\sigma(J)$ holds. 
Taking $\gamma=\mathrm{arg}(\tau)\in [\pi/2,\pi/2]$, we have 
	\begin{equation}\label{Eq:c} c-1=e^{2i\gamma} \mathrm{\;and\;} c=2e^{i\gamma}\cos \gamma. \end{equation} 
Inserting Eq.~(\ref{Eq:c}) into Eq.~(\ref{Eq:rew}) provides
	\[ \mathrm{det}\left ( \frac{(e^{-i\gamma}\lambda)-(e^{-i\gamma}\lambda)^{-1}}{2} -\frac{\kappa}{\sqrt{\kappa^2+q^2(k)}} P \right)=0. \]
We should remark that since $1\in\sigma(P)$ with simple multiplicity, then $1\in \kappa/\sqrt{\kappa^2+q^2(k)} \sigma(P)$ if and only if $q(k)=0$. 
and also remark that since $-1\in\sigma(P)$ if and only if $G$ is bipartite and $q(k)=0$, 
then $-1\in \kappa/\sqrt{\kappa^2+q^2(k)} \sigma(P)$ if and only if $q(k)=0$ and $G$ is bipartite. 
Put $M_+=|E|-|V|+\bs{1}_{\{q(k)=0\}}$ and $M_-=|E|-|V|+\bs{1}_{\{q(k)=0\;and \;G \;is \;bipartite\}}$. 
Let $b(G)$ be the number of the essential cycles of $G$. 
So the following statement holds: 
\begin{center}
\begin{tabular}{|c|c|c|c|c|}\hline
 & $G$ is a tree & $b(G)=1$, bipartite & $b(G)=1$, non-bipartite & $b(G)\geq 2$ \\ \hline
$q=0$ & $M_\pm\leq 0$ & $M_\pm=1$ & $M_+=1$, $M_-=0$ & $M_\pm>0$ \\ \hline
$q\neq 0$ & $M_\pm\leq 0$ & $M_\pm=0$ & $M_\pm=0$ & $M_\pm>0$ \\ \hline
\end{tabular}
\end{center}
Thus we can complete the proof by applying the last half of the proof of Proposition~2. 
\end{proof}
\begin{corollary}
	Set $A=\left\{k^2\in [0,\infty): 1\in e^{i(kL+\mathrm{arg}(\tau))}\varphi_1^{-1}\left( \frac{\kappa}{\kappa+q^2(k)}\sigma(P) \right)\right\}$. 
        The spectrum of the quantum graph walk is expressed as follows:  
        \[ \begin{cases}
           A & \text{: $G$ is a tree} \\
           A\cup \left\{ \left(\frac{2n\pi}{L}\right)^2: n\in \mathbb{N} \right\} \cup \left\{ \left(\frac{(2n+1)\pi }{L}\right)^2: n\in \mathbb{N} \right\}  & \text{: $G$ has at least one cycle} \\
           \end{cases}
        \]
\end{corollary}
\subsection{Example 3: Positive support of the Grover walk}
Recently there are some trials to apply QW to graph isomorphism problems. 
For a matrix $M$, the positive support of $M$, $M^+$, is denoted by 
	\[ (M^+)_{i,j}= \begin{cases} 1 & \text{: $(M)_{i,j}\neq 0$} \\ 0 & \text{: otherwise} \end{cases} \]
It is suggested that the spectra of $(U^3)^+$, which is the positive support of the cube of the Grover matrix $U^3$, 
outperforms distinguishing strongly regular graphs in \cite{Emms}. 
It seems to be that not only the Grover matrix $U$ itself but the positive support $(U^n)^+$ of its $n$-th power is an important operator of a graph.
For a first step, in this section, we consider $\sigma(U^3)^+$ applying Lemma~\ref{keylemma}. 
To do so we prepare the following fact: 
	\begin{proposition}\cite{HKSS:JMI} \label{Pro:HKSS_JMI}
        Let $G$ be a connected and $\kappa$-regular graph with $g(G)\geq 5$. Then we have 
        \begin{equation} (U^3)^+=(U^+)^3+ {}^TU^+. \end{equation}
	\end{proposition}
We can rewrite $U^+$ as follows: 
	\begin{lemma}
	When we take $m_A(e)=1$, $m_V(u)=1$ and $w(e)=1$, then we have 
	\[U^+=S(d^*d-1_A)\]
        \end{lemma}
The results on $(U)^+$ and $(U^2)^+$ have already appeared in \cite{HKSS:JMI,HKSS:PJMI}. In this paper, we newly obtain $(U^3)^+$ case as follows. 
\begin{proposition}
Assume that $G$ is a connected and $\kappa$-regular graph with $g(G)\geq 5$. 
Let $M$ be the adjacency matrix of $G$. 
We set 
        \begin{align}
		s_{\pm}^{(1)}(x) &= \frac{x}{2} \pm \frac{\sqrt{x^2-4\kappa+4}}{2}, \\
                s_{\pm}^{(2)}(x) &= \frac{x^2-2 \kappa+4}{2}\pm \frac{x\sqrt{x^2-4\kappa+4}}{2}
	\end{align}
	and 
       \begin{multline} 
        	s_{\pm}^{(3)}(x)=\frac{x (x^2+4-3\kappa)}{2}  \\ \pm \frac{1}{2}\sqrt{x^6+2 (2-3\kappa) x^4+\left(13\kappa^2-24\kappa+16\right) x^2-8 (\kappa-1)(\kappa^2-2\kappa+2)}. 
        \end{multline}
        Then we have 
	\begin{align}
        \sigma\left((U)^+\right) &= \{s_\pm^{(1)} (\lambda):\lambda\in \sigma(M)\} \cup \{1\}^{|E|-|V|} \cup \{-1\}^{|E|-|V|}\\
        \sigma\left((U^2)^+\right) &= \{s_\pm^{(2)} (\lambda):\lambda\in \sigma(M)\} \cup \{2\}^{2(|E|-|V|)}\\
        \sigma\left((U^3)^+\right) &=\{s_\pm^{(3)} (\lambda):\lambda\in \sigma(M)\} \cup \{2\}^{|E|-|V|} \cup \{-2\}^{|E|-|V|}
        \end{align}
\end{proposition}

\begin{proof}
The proofs of cases of $(U)^+$ and $(U^2)^+$ can be seen in \cite{HKSS:JMI,HKSS:PJMI} 
So we provide only the proof for the $(U^3)^+$ case. 
Since ${}^TU^+=SU^+S$ holds, by Eq.~(\ref{Eq:close}), we have 
 \begin{align} 
 	\begin{bmatrix} {}^TU^+d^* & {}^TU^+ Sd^* \end{bmatrix}
        	&= S \begin{bmatrix} U^+Sd^* & U^+ d^* \end{bmatrix} \notag \\
                &= S \begin{bmatrix} Sd^* & d^* \end{bmatrix} \widetilde{T}' \notag\\
                &= \begin{bmatrix} d^* & Sd^* \end{bmatrix} \widetilde{T}', \label{Eq:trans}
 \end{align}
where 
	\begin{align*}
        \widetilde{T} = \begin{bmatrix} 0 & -1_V \\ (k-1)1_V & M \end{bmatrix} \mathrm{\;\;and\;\;}
        \widetilde{T}'=\begin{bmatrix} 0 & 1_V \\ 1_V & 0 \end{bmatrix} \widetilde{T}\begin{bmatrix} 0 & 1_V \\ 1_V & 0 \end{bmatrix}. 
        \end{align*}
Using Eq.~(\ref{Eq:close}) again,  
	\begin{equation}\label{Eq:cube} \begin{bmatrix} (U^+)^3d^* & (U^+)^3 Sd^* \end{bmatrix}= \begin{bmatrix} d^* & Sd^* \end{bmatrix} \widetilde{T}^3 \end{equation}
Inserting Eqs.~(\ref{Eq:trans}), (\ref{Eq:cube}) into Proposition~(\ref{Pro:HKSS_JMI}) provides 
	\begin{equation}\label{Eq:cubeposi}
        	\begin{bmatrix} (U^3)^+ d^* & (U^3)^+ Sd^* \end{bmatrix}
                	= \begin{bmatrix} d^* & Sd^* \end{bmatrix} (\widetilde{T}^3+\widetilde{T}')
        \end{equation}
Now we consider the spectra of $\widetilde{T}^3+\widetilde{T}'$. 
Since $M=M^*$, $M$ is decomposed by $M=\sum_{\lambda\in \sigma(M)} \lambda \Pi_\lambda$, where $\Pi_\lambda$ is the orthogonal projection onto the eigenspace of eigenvalue $\lambda$. 
By Eq.~(\ref{Eq:cubeposi}), 
	\begin{align}
        \widetilde{T}^3+\widetilde{T}'=\sum_{\lambda\in \sigma(M)}\Lambda(\lambda) \otimes \Pi_\lambda. 
        \end{align}
Here 
	\[ \Lambda(\lambda)=\begin{bmatrix} -\lambda (\kappa-2) & -\lambda^2+2\kappa-2 \\ (\kappa-1)(\lambda^2-\kappa+1)-1 & \lambda(\lambda^2-2\kappa+2)\end{bmatrix}\]
A simple computation gives  
	\[ \sigma(\Lambda(\lambda))=\{s^{(3)}_\pm (\lambda)\}. \]
From the above, we obtain $2|V|$ eigenvalues of $(U^3)^+$. 
All the $2|V|$ eigenvectors in $\mathcal{L}\equiv d^* (C^0(G))+Sd^*(C^0(G))$ are 
	\[ \left\{(\lambda^2-2\kappa+2)d^*f_\lambda-\left(\lambda(\kappa-2)-s_{\pm}^{(3)}(\lambda)\right) Sd^*f_\lambda: \lambda\in \sigma(M) \right\}.  \]
Here $f_\lambda\in \mathrm{ker}(\lambda-M)$. 
Note that since $(U^3)^+$ is no longer ensured its regularity, there may exist $\lambda\in \sigma(M)$ such that $s_+(\lambda)=s_-(\lambda)\equiv s(\lambda)$. 
In such a case, the geometric multiplicity of $s(\lambda)$ is $\mathrm{dim}\mathrm{ker}(T-\lambda)$ 
while the algebraic multiplicity is $2\mathrm{dim}\mathrm{ker}(T-\lambda)$. 

Next we consider the other invariant subspace $\mathcal{L}^\bot$. 
It holds that $\mathcal{L}^\bot=\mathrm{ker}(d)\cap \mathrm{ker}(dS)=(\mathrm{ker}(d) \cap \mathcal{H}_+) \oplus (\mathrm{ker}(dS) \cap \mathcal{H}_+)$. 
For $\psi\in \mathrm{ker}(d)\cap \mathcal{H}_\pm$, we can check that $(U^+)^3 \psi = \mp 2\psi$. 
Therefore the algebraic multiplicities of eigenvalues $\pm 2$ are $|E|-|V|$. 
\end{proof}

Let $Z_G^{(j)}(u)=\prod_{\lambda\in\sigma((U^j)^+)} (1-u\lambda)^{-1}$, ($j\in\{1,2,3\}$). 
The following figure depict the poles of $Z_G^{(j)}(u)$ for the Petersen graph. 
\begin{figure}[htbp]
 \begin{minipage}{0.31\hsize}
  \begin{center}
   \includegraphics[width=40mm]{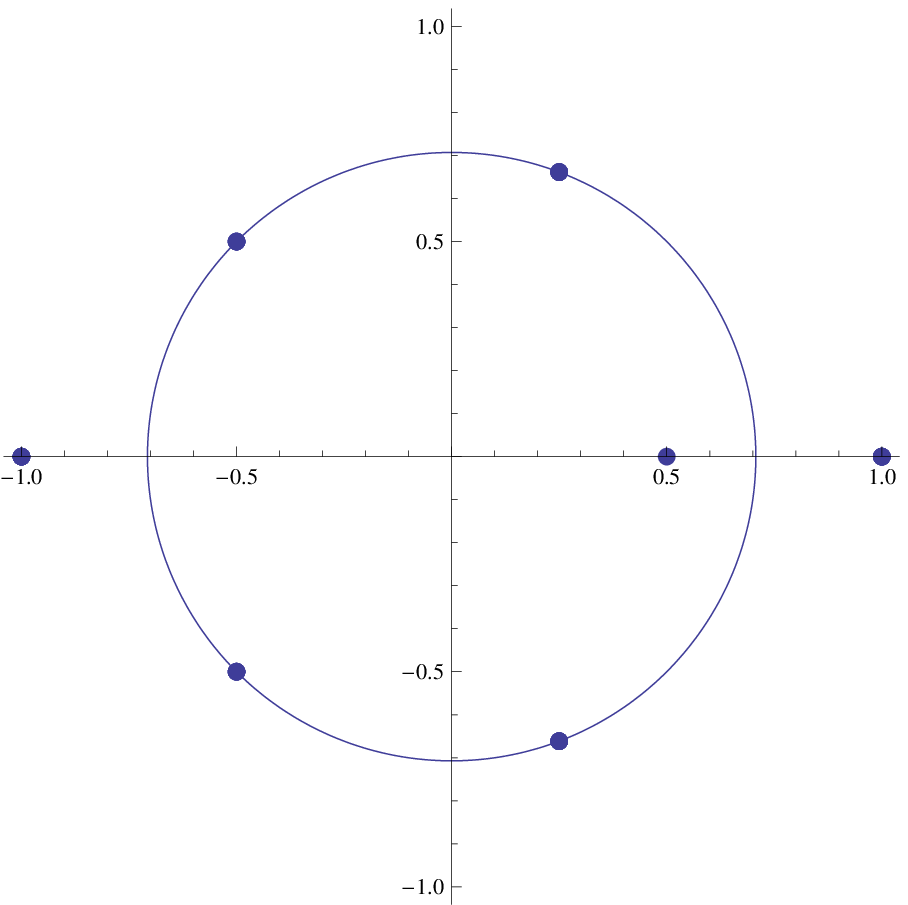}
  \end{center}
  \caption{Poles of $U^+$}
  \label{fig:one}
 \end{minipage}
 \begin{minipage}{0.31\hsize}
 \begin{center}
  \includegraphics[width=30mm]{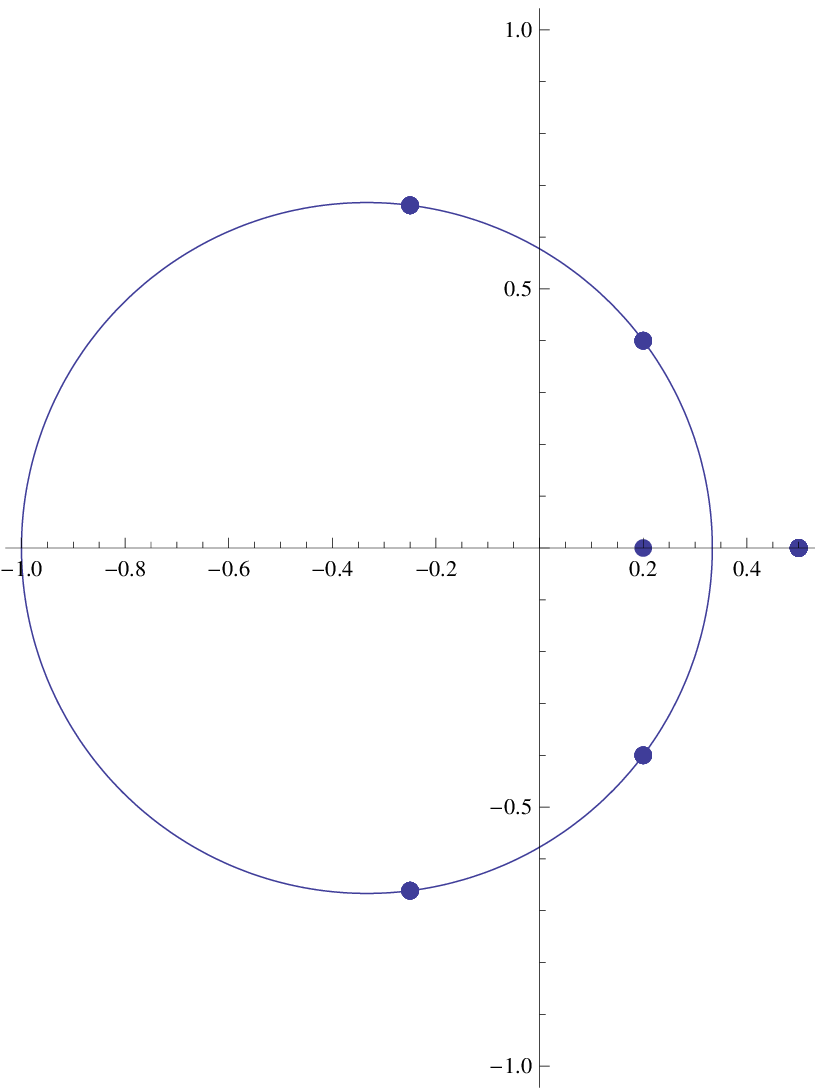}
 \end{center}
  \caption{Poles of $(U^2)^+$}
  \label{fig:two}
 \end{minipage}
 \begin{minipage}{0.31\hsize}
 \begin{center}
  \includegraphics[width=50mm]{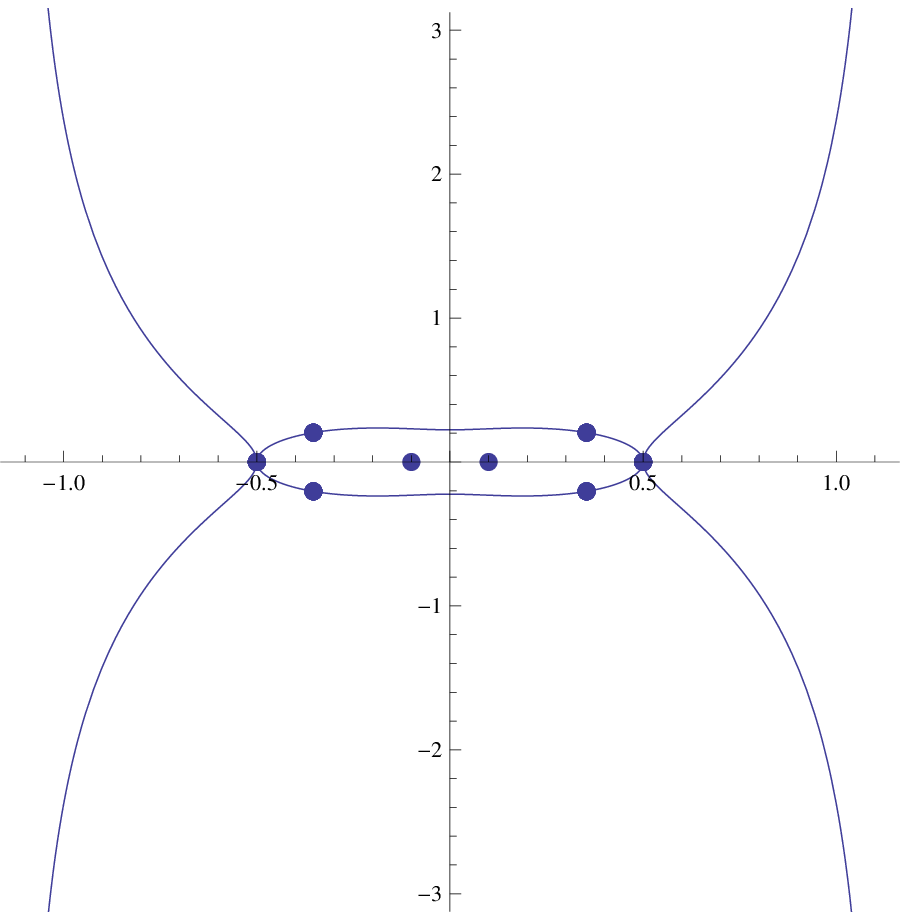}
 \end{center}
  \caption{Poles of $(U^3)^+$}
  \label{fig:three}
 \end{minipage}
\end{figure}

\noindent \\

\noindent{\bf Acknowledgment} 
ES thanks to the financial supports of the Grant-in-Aid for Young Scientists (B) and of Scientific Research (B) Japan Society for the
Promotion of Science (Grants No.25800088, No.23340027).

%


\begin{small}
\bibliographystyle{jplain}

\begin{thebibliography}{99}

\bibitem{Am}
A. Ambainis, 
Quantum walks and their algorithmic applications, 
Int. J. Quantum Inf. {\bf 1} (2003), pp.507-518.


\bibitem{Sze}
M. Szegedy, 
Quantum speed-up of Markov chain based algorithms, 
Proc. 45th IEEE Symposium on Foundations of Computer Science (2004), pp.32--41. 

\bibitem{Emms}
D. Emms, E. R. Hancock, S. Severini, R. C. Wilson, 
A matrix representation of graphs and its spectrum as a graph invariant, 
Electr. J. Combin. {\bf 13}, R34 (2006).

\bibitem{Exner}
P. Exner, P. Seba, 
Free quantum motion on a branching graph, 
Rep. Math. Phys. {\bf 28} (1989), pp.7-26.

\bibitem{HKSS:JFA}
Yu. Higuchi, N. Konno, I. Sato and E. Segawa, 
Spectral and asymptotic properties of Grover walks on crystal lattices, 
Journal of Functional Analysis {\bf 267} (2014), pp.4197-4235. 

\bibitem{HKSS:JMI}
Yu. Higuchi, N. Konno, I. Sato, E. Segawa, 
A note on the discrete-time evolutions of quantum walk on a graph, 
Journal of Math-for-Industry {\bf 5} (2013), pp.103--109.

\bibitem{HKSS:PJMI}
Yu. Higuchi, N. Konno, I. Sato and E. Segawa, 
A remark on zeta functions of finite graphs via quantum walks, 
Pacific Journal of Math-for-Industry {\bf 6} (2014), pp.73--81. 

\bibitem{HKSS:YMJ}
Yu. Higuchi, N. Konno, I. Sato and E. Segawa, 
Quantum graph walks I: mapping to quantum walks, 
Yokohama Mathematical Journal {\bf 59} (2013), pp.33-55.

\bibitem{GG}
C. Godsil and K. Guo, 
Quantum walks on regular graphs and eigenvalues, 
Electron. J. Combin. {\bf 18} (2011) P165. 

\bibitem{Gudder}
S. Gudder, 
Quantum Probability, 
Academic Press Inc., CA, 1988.

\bibitem{Severini}
S. Severini, 
On the digraph of a unitary matrix, 
SIAM Journal on Matrix Analysis and Applications {\bf 25} (2003), pp.295-300.

\bibitem{Tan}
G. Tanner, 
From quantum graphs to quantum random walks, 
In: Non-linear Dynamics and Fundamental Interactions NATO Science Series II: Mathematics, Physics
and Chemistry {\bf 213} (2006), pp.69--87. Springer, Netherlands.

\end{thebibliography}

\end{small}


\end{document}